\documentclass[a4paper,12pt,leqno]{article}

\usepackage[T2A]{fontenc}
\usepackage{amsfonts,amsmath,amssymb,euscript}
\usepackage{comment}
\usepackage[dvips]{graphicx}
\usepackage{subfigure}
\usepackage{expdlist}

\newtheorem{theorem}{Theorem}[section]
\newtheorem{lemma}[theorem]{Lemma}

\numberwithin{equation}{section}

\newcommand{\set}[1]{\{ #1 \}} 
\newcommand{\setst}[2]{\{ #1 \mid #2 \}} 

\newcommand{\abs}[1]{| #1 |}
\newcommand{\val}[1]{\mathop{\rm val}\nolimits ( #1 )}

\providecommand{\R}{\mathbb{R}} \providecommand{\Z}{\mathbb{Z}}

\renewcommand{\phi}{\varphi}

\renewcommand{\hat}{\widehat}
\renewcommand{\bar}{\overline}
\renewcommand{\div}{\mathop{\rm div}\nolimits}

\newcommand{\calC}{\mathcal{C}}

\newcommand{\calN}{\mathcal{N}}
\newcommand{\calR}{\mathcal{R}}

\newcommand{\deltain}{\delta^{\rm in}}
\newcommand{\deltaout}{\delta^{\rm out}}

\newcommand{\qed}{\hfill$\square$}

\newenvironment{proof}{\noindent{\bf Proof}~}%
{\qed\medskip}

\newenvironment{numitem}{%
   \refstepcounter{equation}%
   \begin{enumerate} \compact%
      \item[(\theequation)]}{%
   \end{enumerate}
}

\newenvironment{nameditem}[1]{%
   \begin{description}%
      \rm \item[(#1)] \em%
}{%
   \end{description}%
}

\newcommand{\refeq}[1]{(\ref{eq:#1})} 
\newcommand{\refth}[1]{Theorem~\ref{th:#1}} 
\newcommand{\reflm}[1]{Lemma~\ref{lm:#1}} 
\newcommand{\refsec}[1]{Section~\ref{sec:#1}} 
\newcommand{\refsubsec}[1]{Subsection~\ref{subsec:#1}} 

\newcommand{\Xcomment}[1]{}

\makeatletter
\renewcommand{\section}{\@startsection{section}{1}{0pt}%
{-3.5ex plus -1ex minus -.2ex}{2.3ex plus .2ex}%
{\normalfont\large\bf}} \makeatother

 \makeatletter
\renewcommand{\subsection}{\@startsection{subsection}{2}{0pt}%
{-3.0ex plus -1ex minus -.2ex}{1.5ex plus .2ex}%
{\normalfont\normalsize\bf}}
 \makeatother

\begin{document}

\title{On Weighted Multicommodity Flows \\in Directed Networks}

\author{Maxim A.~Babenko\thanks{Moscow State University; Leninskie
Gory, 119991 Moscow, Russia; email: \texttt{maxim.babenko@gmail.com}\,. A part
of this research was done while this author was visiting Institut f\"ur
Diskrete Mathematik, Univ. Bonn.}
 \and
Alexander V.~Karzanov\thanks{Inst. for System Analysis of the RAS; 9, Prospect
60 Let Oktyabrya, 117312 Moscow, Russia; email: \texttt{sasha@cs.isa.ac.ru}.
Supported by RFBR grant 10-01-9311-CNRSL\_\,a.}
 }


\date{}

\maketitle
 \begin{abstract}
Let $G = (VG, AG)$ be a directed graph with a set $S \subseteq
VG$ of terminals and nonnegative integer arc capacities $c$. A feasible
\emph{multiflow} is a nonnegative real function $F(P)$ of ``flows'' on
paths~$P$ connecting distinct terminals such that the sum of flows through each
arc $a$ does not exceed $c(a)$. Given $\mu \colon S \times S \to \R_+$, the
\emph{$\mu$-value} of $F$ is $\sum_P F(P) \mu(s_P, t_P)$, where $s_P$ and $t_P$
are the start and end vertices of a path $P$, respectively.

Using a sophisticated topological approach, Hirai and Koichi showed that the
maximum $\mu$-value multiflow problem has an integer optimal solution when
$\mu$ is the distance generated by subtrees of a weighted directed tree and
$(G,S,c)$ satisfies certain Eulerian conditions.

We give a combinatorial proof of that result and devise a strongly polynomial
combinatorial algorithm. \medskip

\end{abstract}

\section{Introduction}

\subsection{Multiflows in directed networks}

We use standard terminology of graph and flow theory. For a digraph $G$, the
sets of its vertices and arcs are denoted by $VG$ and $AG$, respectively. A
similar notation is used for paths, cycles, and etc. For $X \in VG$, the set of
arcs of $G$ entering (resp. leaving) $X$ is denoted by $\deltain_G(X)$ (resp.
$\deltaout_G(X)$). When $X = \set{v}$, we write $\delta^\bullet_G(v)$ for
$\delta^\bullet_G(\{v\})$. When $G$ is clear from the context, it is omitted
from notation. Also for a set $A$ and a singleton $a$, we will write $A-a$ for
$A\setminus\{a\}$, and $A\cup a$ for $A\cup\{a\}$.

A \emph{directed network} is a triple $\calN = (G, S, c)$ consisting of a
digraph~$G$, a set $S \subseteq VG$ of \emph{terminals}, and \emph{integer} arc
\emph{capacities} $c \colon AG \to \Z_+$. Vertices in $VG - S$ are called
\emph{inner}. A directed path in $G$ is called an \emph{$S$-path} if its
endvertices are distinct elements of $S$. A \emph{multiflow} $F$ is a function
assigning a nonnegative real number, or \emph{flow}, to each $S$-path. A
multiflow $F$ is called \emph{feasible} if for each arc $a \in AG$, the sum of
flows assigned to $S$-paths going through $a$ does not exceed $c(a)$. The
(total) \emph{value} of $F$ is the sum of flows over all $S$-paths $P$:
 \begin{equation}
\label{eq:mf_value}
    \val{F} := \sum\nolimits_{P} F(P).
 \end{equation}
Sometimes (e.g., in~\cite{IKN}) such multiflows are called \emph{free} to
emphasize the fact that \emph{any} pair of terminals is allowed to be connected
by nonzero flows.

The following \emph{maximum (fractional) multiflow problem} is well known:
 \begin{nameditem}{\textbf{MF}}
Given a directed network $\calN = (G, S, c)$, find a feasible multiflow $F$ of
maximum value.
 \end{nameditem}
The problem in which one is asked for maximizing among the \emph{integer}
multiflows is denoted  by \textbf{IMF}.

For general directed networks $\calN$, problem~\textbf{IMF} is NP-hard already
for $\abs{S} = 2$ \cite{FHW}. Tractable cases have been revealed for networks
obeying a certain conservation property. More precisely, $c$ (or $\calN$) is
called \emph{Eulerian} at a vertex $v$ if $c(\deltain(v)) = c(\deltaout(v))$.
(For a function $f:A\to\R$ and a subset $A'\subseteq A$, we write $f(A')$ for
$\sum(f(a)\colon a\in A')$.) When $c$ is Eulerian at all inner vertices (resp.
at all vertices), the network $\calN$ is called \emph{inner} (resp.
\emph{totally}) \emph{Eulerian}. By a \emph{cut} in $G$ we mean a pair of
nonempty subsets $(X, \bar X)$, where $X \subset VG$ and $\bar X := VG - X$. It
is called an \emph{$(S_1,S_2)$-cut} if $S_1 \subseteq X$ and $S_2 \subseteq
\bar X$. When it is not confusing, we may refer to the arc sets
$\deltaout_G(X)$ and $\deltain_G(X)$ as cuts as well.

The following result signifies the importance of inner Eulerian networks:

\begin{theorem} [Lomonosov (unpublished, 1978), Frank \cite{frank}]
\label{th:free_min_max}
    Let $\calN=(G,S,c)$ be an inner Eulerian directed network.
    Then there exists an integer maximum feasible multiflow~$F$ in $\calN$. It satisfies
    $$
        \val{F} = \sum\nolimits_{t \in S} c(\deltaout(X_t)),
    $$
    where for each $t \in S$, ~$(X_t, \bar X_t)$ is a minimum capacity
    $(t, S-t)$-cut in~$\calN$.
\end{theorem}

Therefore, \textbf{MF} and \textbf{IMF} have the same optimal value for an
inner Eulerian network, and this value can be found in strongly polynomial time
(by computing a minimum $(t, S-t)$-cut for each $t\in S$). Ibaraki, Karzanov,
and Nagamochi~\cite{IKN} devised a ``divide-and-conquer'' method that computes
an integer maximum multiflow in such a network in $O((MF(n,m) + mn) \cdot \log
\abs{S} + mn^2)$ time. (Hereinafter $n := \abs{VG}$, $m := \abs{AG}$, and
$MF(n', m')$ denotes the complexity of a max-flow computation in a directed
network with $n'$ vertices and $m'$ arcs.) The latter complexity was improved
to $O((MF(n,m) + mn \log (n^2/m)) \cdot \log \abs{S})$ in~\cite{BK-07}.

\subsection{Weighted multiflows}

A generalization of \textbf{MF} involves weights between terminals. More
precisely, given a \emph{weighting} $\mu \colon S \times S \to \R_+$, the
\emph{$\mu$-value} of a multiflow $F$ is
\begin{equation}
\label{eq:mf_mu_value}
    \val{F,\mu} := \sum\nolimits_P \mu(s_P, t_P) F(P),
\end{equation}
where the sum is over all $S$-paths $P$, and $s_P$ and $t_P$ denote the start
and end vertices of~$P$, respectively. We may assume that $\mu(s,s) = 0$ for
all $s \in S$.

Replacing \refeq{mf_value} by \refeq{mf_mu_value}, we obtain the weighted
counterpart of \textbf{MF}:
  \begin{nameditem}{$\mu$-\textbf{MF}}
    Given $\calN$ and $\mu$ as above, find a feasible multiflow $F$ of maximum $\mu$-value.
  \end{nameditem}
The integer strengthening of $\mu$-\textbf{MF} is denoted by
$\mu$-\textbf{IMF}. When $\mu(s,t)=1$ for all $s\ne t$, ~$\mu$-\textbf{MF}
turns into \textbf{MF}, and $\mu$-\textbf{IMF} into \textbf{IMF}.

\subsection{Tree-induced weights} \label{subsec:tree_induced}

It has been shown that problem $\mu$-\textbf{IMF} has a rather wide spectrum of
tractable cases. The simplest case is $S = \set{s, t}$, ~$\mu(s,t) = 1$ and
$\mu(t,s) = 0$; then $\mu$-\textbf{IMF} becomes the standard maximum flow
problem with arbitrary integer capacities. A representative well-solvable class
has been found in connection with the so-called \emph{directed multiflow
locking problem}:
    \begin{nameditem}{\textbf{DMLP}}
Given a directed network $\calN=(G,S,c)$ and a collection $\calC\subseteq 2^S$,
find a feasible multiflow $F$ in $\calN$ that locks simultaneously all members
of $\calC$.
  \end{nameditem}
Here $F$ is said to \emph{lock} a subset $A\subset S$ if the sum of values
$F(P)$ over the $S$-paths $P$ going from $A$ to $S-A$ is maximum possible, i.e.
it is equal to the minimum capacity of an $(A,S-A)$-cut in $\calN$. A
collection $\calC\subseteq 2^S$ is called \emph{lockable} if \textbf{DMLP} has
a solution for all $(G,c)$ (with $S$ fixed). Important facts are given in the
following

\begin{theorem}[\cite{IKN}]
  \label{th:lock}
$\calC\subseteq 2^S$ is lockable if and only if $\calC$ is \emph{cross-free},
i.e. for any $A,B\in\calC$, at least one of the following holds: $A\subseteq
B$, $B\subseteq A$, $A\cap B=\emptyset$, $A\cup B=S$. Moreover, if $\calC$ is
cross-free and $\calN$ is inner Eulerian, then \textbf{DMLP} has an integer
solution. Such a solution can be found in $O((MF(n,m) + mn) \cdot \log \abs{S}+
mn^2)$ time.
   \end{theorem}

This gives rise to the following tractable cases of $\mu$-\textbf{IMF}. Given
$\calC\subseteq 2^S$, take an arbitrary function $\ell:\calC\to\R_+$. For
$s,t\in S$, define
  \begin{equation} \label{eq:mu_ell}
\mu_\ell(s,t):=\sum(\ell(A)\colon A\in\calC, ~s\in A\not\ni t).
  \end{equation}
Suppose that $\calC$ is cross-free and $\calN$ is inner Eulerian, and let $F$
be an integer solution to \textbf{DMLP} (existing by \refth{lock}). Then $F$ is
simultaneously an optimal solution to $\mu_\ell$-\textbf{IMF} for every
$\ell:\calC\to\R_+$; this can be easily concluded from the fact that $F$
saturates minimum capacity $(A,S-A)$-cuts in $\calN$ for all $A\in\calC$.

A cross-free collection $\calC$ can be represented by use of a directed tree
$T=(VT,AT)$ (a digraph whose underlying undirected graph is a tree); namely:
  \begin{numitem} \label{eq:lam-tree}
there is a bijection $\beta:\calC\to AT$ and a map $\gamma:S\to VT$ such that
for each arc $a=(u,v)\in AT$ and for $A:=\beta^{-1}(a)$, the set of terminals
$s\in S$ whose image $\gamma(s)$ occurs in the component of $T-a$ containing
$u$ is exactly $A$.
  \end{numitem}
Under this correspondence, we may interpret $\ell:\calC\to\R_+$ as a
\emph{length} function on the arcs of $T$, keeping the same notation:
$\ell(a):=\ell(A)$ for $a\in AT$ and $A:=\beta^{-1}(a)$. These arc lengths
induce \emph{distances} $d=d_\ell$ on $VT$ in a natural way:
  \begin{numitem} \label{eq:distT}
for $x, y \in VT$, define $d(x,y)$ to be the sum of $\ell$-lengths of
\emph{forward} arcs in the simple path from $x$ to $y$ in $T$.
  \end{numitem}
(This path may contain both forward and backward arcs. If there are no forward
arcs, we have $d(x,y)= 0$.) One can see that $\mu=\mu_\ell$ figured
in~\refeq{mu_ell} satisfy
  \begin{equation} \label{eq:metr-mu}
  \mu(s,t)=d_\ell(\gamma(s),\gamma(t))\qquad s,t\in S.
  \end{equation}

Note that such a $\mu$ satisfies the triangle inequalities
$\mu(s,t)+\mu(t,u)\ge \mu(s,u)$ for all $s,t,u\in S$, i.e. $\mu$ is a
\emph{directed metric (space)}. In light of~\refeq{distT}--\refeq{metr-mu},
$\mu$ is called a \emph{tree-induced} directed metric. \smallskip

Generalizing the above-mentioned integrality results, Hirai and
Koichi~\cite{HK} considered arbitrary weight (or \emph{distance}) functions
$\mu:S\times S\to \R_+$ and gave an exhaustive analysis of the integrality and
``unbounded fractionality'' behavior of problem $\mu$-\textbf{MF} in terms of
$\mu$, for both integer and Eulerian cases.

More precisely, let $T$ be a directed tree with nonnegative arc lengths $\ell$.
Given a network $\calN$ as before, suppose that each terminal $s \in S$ is
associated with some \emph{subtree} (a weakly connected subgraph) $T_s$ of~$T$.
For $s, t \in S$, define $\mu(s,t)$ to be the distance from $T_s$ to $T_t$,
i.e.
$$
    \mu(s,t) := \min \{d(u,v) \colon u \in T_s,\, v \in T_t\}.
$$
(Note that such a $\mu$ need not satisfy triangle inequalities. When each $T_s$
is a single vertex, $\mu$ is specified as in~\refeq{metr-mu}.)

A weight function $\mu$ on $S\times S$ that can be obtained in this way is
called a \emph{tree-induced (directed) distance}, and an appropriate tuple
$\calR = (T, \ell, \set{T_s})$ is called a \emph{(tree) realization} of $\mu$.

For such an $\calR$, we distinguish between three sorts of terminals. We call
$s\in S$ \emph{simple} if $T_s$ consists of a single vertex, \emph{linear} if
$T_s$ is a directed path, and \emph{complex} otherwise. If the whole tree $T$
forms a directed path, then $\calR$ itself is called \emph{linear}.

Hirai and Koichi obtained the following results.
  \begin{theorem}[\cite{HK}]
   \label{th:good_mu}
Let $\calN = (G, S, c)$ be a directed network and let $\mu \colon S \times S \to \R_+$
be a tree-induced distance with a realization $\calR$.
    \begin{enumerate}
\item[\rm (i)] If $\calR$ is linear, then
$\mu$-\textbf{MF} has an integer optimal solution.
\item[\rm (ii)] If $c$ is Eulerian at all inner vertices and all
complex terminals, then $\mu$-\textbf{MF} has an integer optimal solution.
    \end{enumerate}
\end{theorem}

Note that the proof of (i) in \refth{good_mu} given in~\cite{HK} is
constructive; it reduces problem $\mu$-\textbf{MF} in this case to finding a
certain min-cost circulation. The proof in~\cite{HK} is much more involved; it
employs a topological approach based on a concept of tight spans of directed
distance spaces introduced in that paper. (Another nice result in~\cite{HK}
relying on the directed tight span approach asserts that if a distance $\mu$ is
not tree-induced, then $\mu$-\textbf{MF} has unbounded fractionality in the
totally Eulerian case, i.e. there is no positive integer $k$ such that
$\mu$-\textbf{MF} admits a $\frac1k$-integer solution for every totally
Eulerian network and this $\mu$.)

  \medskip
In this paper we devise an efficient combinatorial algorithm that constructs an
integer optimal solution to problem $\mu$-\textbf{MF} under the conditions as
in~(ii) of \refth{good_mu}; this yields an alternative (and relatively simple)
proof of assertion~(ii). Our method extends the divide-and-conquer approach
of~\cite{IKN}; it is described in \refsec{algo}. The algorithm runs in
$O((MF(n,m) + mn \log (n^2/m)) \cdot \log \abs{S})$ time.

\section{Algorithm}
\label{sec:algo}

Let $\calN = (G, S, c)$ and $\mu$ obey the conditions in \refth{good_mu}(ii).

The following convention will allow us to slightly simplify the description of
our algorithm (without loss of generality). In a tree realization $\calR = (T,
\ell, \set{T_s})$ of a distance function $\mu$, let us think of $T$ as an
\emph{undirected} tree with \emph{edge} set $ET$, and assume that each edge
$e=uv$ generates two oppositely directed arcs: one going from $u$ to $v$, and
the other from $v$ to $u$ (yielding a ``directed quasi-tree''). For $a=(u,v)$,
the opposite arc $(v,u)$ is denoted by $\bar a$. The length function $\ell$ is
given on the corresponding arc set, denoted by $AT$ as before, and for $x,y\in
VT$, the distance $d(x,y)$ is defined to be the $\ell$-length of the
corresponding \emph{directed} path from $x$ to $y$. Accordingly, a terminal $s$
is linear if the subtree $T_s$ is a simple undirected path and one of the two
directed paths behind $T_s$ has zero $\ell$-length.

Sometimes, to ensure the desired efficiency of the method, we will be forced to
treat some linear terminals as complex ones (which will never be confusing).

\subsection{Initial reductions} \label{ssec:reduct}

Let $\calR = (T, \ell, \set{T_s})$ be a tree realization of $\mu$. A
\emph{pre-processing stage} of the algorithm applies certain reductions to
$\calR$ (called \emph{initial reductions}).

Choose a linear terminal $s \in S$ (if exists), i.e. $T_s$ is a path in $T$
connecting some vertices $t_1$ and $t_2$, and one of the two directed paths
behind $T_s$, from $t_2$ to $t_1$ say, has zero $\ell$-length. This implies
that
  \begin{equation}
\label{eq:tree_path_dist}
    \mu(x, s) = d(T_x, t_1) \quad\mbox{and}\quad \mu(s, x) = d(t_2, T_x)
    \qquad \mbox{for all $x \in S$}.
\end{equation}

We replace $s$ by a pair of simple terminals as follows. Add to $G$ new
terminals $s_1$ and $s_2$ and arcs $(s, s_1)$ and $(s_2, s)$. The capacities of
these arcs are chosen to be sufficiently large and to make the network Eulerian
at $s$. Denote the resulting digraph by~$G'$ and the network by $\calN' :=
(G',S', c')$, where $S' := (S - s) \cup \set{s_1, s_2}$. We modify $\calR$ into
the tuple $\calR'$ with the same tree~$T$ by setting $T_{s_1} := \set{t_1}$ and
$T_{s_2} := \set{t_2}$. This gives new distance $\mu' \colon S' \times S' \to
\R_+$.

We claim that the two problems: $\mu$-\textbf{MF} with $(\calN, \calR)$ and
$\mu'$-\textbf{MF} with $(\calN', \calR')$, are essentially equivalent. Indeed,
$\mu'(s_2, s_1) = 0$ (since the $\ell$-length of the directed path from $t_2$
to $t_1$ in $T$ is zero). Therefore, one may consider only those multiflows in
$\calN'$ that are zero on all $s_2$--$s_1$ paths. Any other $S'$-path $P'$ in
$\calN'$ has a natural image (an $S$-path) $P$ in $\calN$. Namely, if $P'$
neither starts at $s_2$ nor ends at $s_1$, then $P = P'$. If $P'$ starts at
$s_2$ (resp. ends at $s_1$), then $P$ is its maximal subpath from $s$ (resp. to
$s$). This gives a one-to-one correspondence between the $S'$-paths $P'$ in
$\calN'$ (excluding $s_2$--$s_1$ ones) and the $S$-paths $P$ in $\calN$, and by
\refeq{tree_path_dist}, the transformation preserves distances: $\mu'(s_{P'},
t_{P'}) = \mu(s_P, t_P)$. We reset $\calN := \calN'$ and $\calR := \calR'$.

Making a sequence of similar reductions, we obtain a situation when
   \begin{enumerate}
    \item[(C1)]
Any terminal in $\calN$ is either simple or complex.
  \end{enumerate}
Four more sorts of reductions are applied to ensure the following additional
properties:
\begin{enumerate}
  \item[(C2)]
Each leaf in $T$ corresponds to some (possibly multiple) simple terminal. (For
otherwise the leaf can be removed from $T$.)
 \item[(C3)]
No inner vertex $v$ of $T$ corresponds to a simple terminal. (For otherwise one
can add to $T$ a new vertex $v'$ and edge $vv'$ with zero $\ell$-length of both
arcs $(v,v'),(v',v)$, and replace the subtree $\{v\}$ by $\{v'\}$ in the
realization of $\mu$.)
 \item[(C4)]
Each inner vertex of $T$ has degree at most 3. (This can be achieved by
splitting inner vertices of bigger degrees in $T$ and by adding additional
edges with zero $\ell$-length of arcs in both directions.)
 \item[(C5)] $T$ has $O(|S|^2)$ vertices.
\end{enumerate}

\noindent To provide~(C5), note that $T$ has $O(|S|)$ leaves (by~(C2)), and
hence it has $O(|S|)$ vertices of degree 3. Consider a vertex $v$ of degree 2
in $T$, and let $e=uv$ and $e'=vw$ be its incident edges. If for any $s\in S$,
the subtree $T_s$ contains either none or both of $e,e'$, then we can merge
$e,e'$ into one edge $uw$ (adding up the corresponding arc lengths), obtaining
a realization of $\mu$ with a smaller tree size. Otherwise $v$ is a leaf of
some $T_s$. Obviously, the number of leaves of $T_s$ does not exceed that of
$T$, so it is estimated as $O(|S|)$. This gives~(C5).

\subsection{Optimality certificate}

Here we establish a sufficient condition that implies optimality of a given
multiflow. We need some additional terminology and notation. A feasible
multiflow $F$ in $\calN$ is said to \emph{saturate} a cut $(X, \bar X)$ in~$G$
if each $S$-path $P$ with $F(P)>0$  meets $\deltain(X)\cup \deltaout(X)$ at
most once, and
$$
    \sum\nolimits_{P\colon e \in AP} F(P)  = c(e)
    \quad
    \mbox{for each arc $e\in\deltaout(X)$}.
$$

\noindent \textbf{Definition.} For an arc~$a = (u, v)$ of~$T$, define $\Pi_a$
to be the set of pairs $(s,t)$ in $S$ such that $\mu(s,t)$ ``feels'' $\ell(a)$,
i.e. $a$ belongs to a minimal directed path that starts in $T_s$ and ends in
$T_t$. (Then $\mu(s,t)$ is just the $\ell$-length of such path.) \medskip

For a multiflow $F$ and a set $\Pi \subseteq S \times S$, let $F[\Pi]$ be the
``restriction'' of $F$ relative to $\Pi$. More precisely, for an $S$-path $P$
in $G$, we define
  $$
    F[\Pi](P) := \begin{cases}
        F(P) & \mbox{if $(s_P, t_P) \in \Pi$}, \\
        0 & \mbox{otherwise}.
    \end{cases}
  $$
A cut $(X, \bar X)$ in~$G$ is called \emph{$\Pi$-separating} if $s \in X\not\ni
t$ holds for each $(s, t) \in \Pi$.

  \begin{lemma}
   \label{lm:opt_cuts}
Let $F$ be a feasible multiflow in $\calN$. Suppose that
  \begin{numitem}
  \label{eq:opt_cuts}
there exists a collection $\{(X_a, \bar X_a)\colon a \in
AT\}$ of cuts in $G$ such that for each $a \in AT$, $(X_a, \bar X_a)$ is a
$\Pi_a$-separating cut saturated by $F$.
  \end{numitem}
Then $F$ is an optimal solution to $\mu$-\textbf{MF}.
  \end{lemma}

  \begin{proof}
For $s, t \in S$, let $f(s,t)$ denote the sum of flows (by $F$) over the paths
from $s$ to $t$ in $G$. Then
    $$
        \val{F,\mu} = \sum\nolimits_{(s,t) \in S\times S} f(s,t) \mu(s,t).
    $$
Also
    $$
        \mu(s,t) = \sum\nolimits_{a \in AT\colon (s,t) \in \Pi_a} \ell(a).
    $$
It follows that
    \begin{equation}
    \label{eq:opt_cuts1}
\val{F,\mu} = \sum\nolimits_{a \in AT} \ell(a) \left(\sum\nolimits_{(s,t) \in
\Pi_a} f(s,t)\right).
    \end{equation}
Consider an arc $a \in AT$. Since $(X_a, \bar X_a)$ is a $\Pi_a$-separating
cut, we have
    \begin{equation}
    \label{eq:opt_cuts2}
\sum\nolimits_{(s,t) \in \Pi_a} f(s,t) \le c(\deltaout(X_a)),
    \end{equation}
Then \refeq{opt_cuts1} and \refeq{opt_cuts2} give
    \begin{equation}
    \label{eq:opt_cuts3}
        \val{F,\mu} \le \sum\nolimits_{a \in AT} \ell(a) \, c(\deltaout(X_a)).
    \end{equation}
Since each cut $(X_a,\bar X_a)$ is saturated by $F$, inequality
\refeq{opt_cuts2} turns into equality, and so does \refeq{opt_cuts3}. Thus,
$\val{F,\mu}$ is maximum, and the lemma follows.
 \end{proof}

Given a problem instance $(\calN, \calR)$, our algorithm will construct an
integer multiflow $F$ that possesses property~\refeq{opt_cuts}, and therefore
$F$ is optimal by \reflm{opt_cuts}. Note that~\refeq{opt_cuts} does not involve
the lengths $\ell$ of arcs in $T$, so $F$ is optimal simultaneously for all
distances $\mu$ induced by arbitrary $\ell$ (when $T$ and $\{T_s\}$ are fixed).

\subsection{Partitioning step}
\label{subsec:partition}

The core of the algorithm consists in the following recursive procedure that
divides the current instance $(\calN, \calR)$ into a pair of smaller ones.

Suppose that $T$ contains an edge $e=v_1v_2$ such that neither $v_1$ nor $v_2$
is a leaf. Let $a:=(v_1,v_2)$. Deletion of $e$ splits $T$ into subtrees $T_1$
and $T_2$ with $v_1 \in VT_1$ and $v_2 \in VT_2$. Define $S_1$ (resp. $S_2$) to
be the set of terminals $s \in S$ such that $T_s$ is entirely contained in
$VT_1$ (resp. in~$VT_2$). Then $S_1 \cap S_2 = \emptyset$ and each terminal in
$S - (S_1 \cup S_2)$ is complex (by properties~(C1),(C3)). Hence $\calN$ is
Eulerian at each vertex in $VG - (S_1 \cup S_2)$. Also from the definition of
$\Pi_a$ it follows that
   $$
    \Pi_a = S_1 \times S_2\quad \mbox{and} \quad \Pi_{\bar a} = S_2 \times S_1.
  $$
Compute an $(S_1,S_2)$-cut $(X_1, X_2)$ of minimum capacity $c(\deltaout(X_1))$
in $G$. Then $(X_1,X_2)$ is $\Pi_a$-separating and $(X_2,X_1)$ is $\Pi_{\bar
a}$-separating. The Eulerianess implies
 $$
    c(\deltaout(X_1)) - c(\deltain(X_1)) =
    \sum\nolimits_{s \in S_1} c(\deltaout(s)) - c(\deltain(s)).
$$
Hence the capacity $c(\deltaout(X_2))=c(\deltain(X_1))$ is minimum among all
$(S_2,S_1)$-cuts in $G$ as well.

We construct two new instances $(\calN_1,\mu_1)$ and $(\calN_2,\mu_2)$ in a
natural way. More precisely, set $\calN_1 := (G_1, S_1', c_1)$, where $G_1$ is
obtained from $G$ by contracting $X_2$ into a new vertex~$z_2$ (and deleting
the loops if appeared), $c_1$ is the restriction of $c$ to the arc set of
$G_1$, and $S_1' := (S \cap X_1) \cup \set{z_2}$. The distance $\mu_1$ is
induced by the tree $T'_1$ obtained from $T$ by contracting the subtree $T_2$
into $v_2$; the arc lengths in $T_1'$ are same as in $T$. (In fact, these
lengths are ignored by the algorithm and they are needed only for our
analysis.) Terminals $s \in S \cap X_1$ are now realized by the subtrees of
$T_1'$ obtained by restricting the subtrees $T_s$ in $\calR$ to $T_1'$. The
terminal $z_2$ is realized by $\set{v_2}$. Let $\calR_1$ denote the
resulting realization of $\mu_1$.

The construction of $\calN_2= (G_2, S_2', c_2),\,\mu_2,\,\calR_2$ is symmetric
(by swapping $1\leftrightarrow 2$).

The algorithm recursively constructs integer optimal multiflows $F_1$ and $F_2$
for $(\calN_1, \calR_1)$ and $(\calN_2, \calR_2)$, respectively. The following
property easily follows from the minimality of $(X_1,X_2)$ and $(X_2,X_1)$:
  \begin{numitem} \label{eq:rec_sat}
the multiflow $F_1$ saturates the cuts $\deltain_{G_1}(z_2)$ and
$\deltaout_{G_1}(z_2)$; similarly, $F_2$ saturates $\deltain_{G_2}(z_1)$ and
$\deltaout_{G_2}(z_1) $.
  \end{numitem}

This property enables us to ``glue'' (or ``aggregate'') $F_1$ and $F_2$ into an
integer multiflow $F$ in $\calN$ which saturates both cuts $(X_1, X_2)$ and
$(X_2, X_1)$ (being ${\Pi_a}$-separating and ${\Pi_{\bar a}}$\,-separating
cuts, respectively). These cuts together with the preimages in $G$ of
corresponding saturated cuts for $F_1$ and $F_2$ give a collection of saturated
cuts for $F$ as required in~\refeq{opt_cuts}, yielding the optimality of $F$ by
\reflm{opt_cuts}.

The above partitioning step reduces the current problem instance to a pair of
smaller ones (in particular, the tree sizes strictly decrease). One easily
checks that conditions (C1)--(C5) (see Section~\ref{ssec:reduct}) are
maintained. Note that for $i\in\{1,2\}$, if $s$ is a complex terminal in $S$
such that the image of $T_s$ in $T'_i$ is different from $\{v_{3-i}\}$, then we
should keep regarding $s$ as a complex terminal in $\calN_i$ (even if this
image is a nontrivial (undirected) path having zero $\ell$-length in one
direction). This is not confusing since the network continues to be Eulerian at
$s$.

The recursion process with a current $T$ stops when each edge in it is incident
to a leaf. Since each inner vertex of $T$ has degree~3 (by~(C4)), only two
cases of $T$ are possible:
  \begin{enumerate}
    \item[\rm (i)] $VT$ consists of two vertices $v_1$ and $v_2$;
    \item[\rm (ii)] $VT$ consists of one inner vertex $v_0$ and three leaves $v_1, v_2, v_3$.
\end{enumerate}

Case~(i) is considered in \refsubsec{two_leaves}, and case~(ii) in
\refsubsec{three_leaves}.

\subsection{Basic step: two vertices}
\label{subsec:two_leaves}

Let $e = v_1v_2$ be the only edge of $T$. Note that the vertices $v_1$ and
$v_2$ may correspond to many terminals in $S$. Let terminals $s_1, \ldots, s_p$
(resp. $t_1, \ldots, t_q$) be realized in  $\calR$ by $\set{v_1}$ (resp.
$\set{v_2}$). Also there may exist a terminals $s$ realized by the whole
tree~$T$; but such an $s$ may be ignored since $\mu(s, t)=\mu(t,s) = 0$ for any
$t \in S$.

Let $S' := \set{s_1, \ldots, s_p}$ and $T' := \set{t_1, \ldots, t_q}$.
Construct an integer maximum $S'$--$T'$ flow, i.e. a function $f \colon AG \to
\Z_+$ with $\val{f} := \sum \left( \div_f(s) : s \in S' \right)$ maximum
subject to $f(a) \le c(a)$ for each $a \in AG$ and $\div_f(v) = 0$ for each $v
\in VG - (S' \cup T')$. Here $\div_f(v)$ denotes the divergence
$f(\deltaout(v)) - f(\deltain(v))$. Then $f$ saturates some $(S',T')$-cut
$(X,\bar X)$. Since capacities~$c$ are Eulerian at all inner vertices, $g := c
- f$ is a $T'$--$S'$ flow. This implies that $\val{g} = c(\bar X,X)$ and that
$g$ saturates the reversed cut $(\bar X, X)$.

We construct $F$ by combining path decompositions of $f$ and $g$. Let
$a:=(v_1,v_2)$. Then $\Pi_a = S' \times T'$, and $(X, \bar X)$ is a
$\Pi_a$-separating cut. The multiflow $F[\Pi_a]$ corresponds to $f$, and
therefore it saturates $(X, \bar X)$. Similarly, $\Pi_{\bar a} = T' \times S'$,
~$(\bar X, X)$ is a $\Pi_{\bar a\,}$-separating cut, the multiflow $F[\Pi_{\bar
a}]$ corresponds to $g$, and therefore it saturates $(\bar X, X)$. This
gives~\refeq{opt_cuts} for $F$.

\subsection{Basic step: three leaves}
\label{subsec:three_leaves}

This case is less trivial. Here $ET$ consists of three edges $e_i=v_iv_0$,
$i=1,2,3$. We denote the arc $(v_i,v_0)$ by $a_i$.

Let us call terminals $s,s'$ in the current network $\calN=(G,S,c)$
\emph{similar} if they are realized by the same subtree of $T$; clearly
$\mu(s,p)=\mu(s',p)$ and $\mu(p,s)=\mu(p,s')$ for any $p\in S$. Suppose that
there are similar \emph{simple} terminals $s,s'$. They correspond to the
singleton $\{v_i\}$ for some $i\in\{1,2,3\}$ (in view of (C3)). The fact that
$v_i$ is a leaf of $T$ provides the triangle inequality $\mu(p,s)+\mu(s,q)\ge
\mu(p,q)$ for any $p,q\in S$, and similarly for $s'$. Due to this, we can
identify $s,s'$ in $G$ into one terminal (corresponding to $\{v_i\}$) without
affecting the problem in essence.

Thus, we may assume that for each $i=1,2,3$, there is exactly one terminal,
$s_i$ say, corresponding to $\{v_i\}$. Let $S':=\{s_1,s_2,s_3\}$. Note that
each terminal $s\in S-S'$ is (regarded as) complex, and $c$ is Eulerian at $s$.
We partition $S-S'$ into subsets $S_1,S_2,S_3,S_{12},S_{13},S_{23}$, where
$S_i$ (resp. $S_{ij}$) consists of the (similar) terminals corresponding to the
subtree of $T$ induced by the edge $e_i$ (resp. by the pair $\{e_i,e_j\}$).

Suppose we ignore the terminals in $S-S'$, by considering the network $\calN'
:= (G, S', c)$. This network is inner Eulerian since $\calN$ is Eulerian within
$S-S'$. Using the algorithm from~\cite{BK-07}, we find an optimal multiflow $F$
to problem \textbf{IMF} for $\calN'$ with unit distance for each pair
$(s_i,s_j)$, $i\ne j$. Also for $i=1,2,3$, we find a minimum capacity $(s_i,\,
S' - s_i)$-cut $(X_i, \bar X_i)$ in $\calN'$. They can be chosen so that the
sets $X_1,X_2,X_3$ are pairwise disjoint. Also one may assume that each path
$P$ with $F(P)>0$ is simple and has no intermediate vertex in $S'$. Then $F$
yields a solution to $(\calN',\mu')$, where $\mu'$ is the restriction of $\mu$
to $S'\times S'$. Since $\calN'$ is inner Eulerian and in view of
\refth{free_min_max}, $F$ saturates both cuts $(X_i, \bar X_i)$ and $(\bar X_i,
X_i)$ for each $i$. Associating such cuts to the arcs $a_i,\bar a_i$ results
in~\refeq{opt_cuts}. Then $F$ is optimal by~\reflm{opt_cuts}.

Next we return to $\calN$ as before. The above multiflow $F$ need not be
optimal for $(\calN,\mu)$ since cuts $(X_i, \bar X_i)$ may not be
$\Pi_{a_i}$-separating for some $i$. Our aim is to improve $F,X_1, X_2, X_3$ so
as to ensure~\refeq{opt_cuts}.

More precisely, we are looking for subsets $X_i' \subseteq X_i$, $i = 1, 2, 3$,
and a  multiflow $F'$ such that:
\begin{numitem}
\label{eq:3opt}
    \begin{enumerate}
        \item[(i)]
        $S_i \subset VG - (X_j \cup X_k)$ and $S_{ij} \subset VG - X_k$ for
        any distinct $i,j,k$;
        \item[(ii)]
        for $i = 1, 2, 3$, the cuts $(X_i',\bar X_i')$ and $(\bar X_i',X_i')$
        are saturated by~$F'$;
        \item[(iii)]
each path $P$ with $F'(P)>0$ connects either $s_i$ and $s_j$, or $s_i$ and
$S_j$, or $s_i$ and $S_{jk}$, where $i,j,k$ are distinct.
    \end{enumerate}
 \end{numitem}
By~\refeq{3opt}(i), the cut $(X_i',\bar X_i')$ is $\Pi_{a_i}$-separating, $i =
1, 2, 3$. In their turn \refeq{3opt}(ii),(iii) imply that $F'[\Pi_{a_i}]$
saturates $(X_i',\bar X_i')$. Then $F'$ is optimal by \reflm{opt_cuts}.

We construct the desired $X'_i$ and $F'$ as follows. For $i = 1, 2, 3$, let
$Q_i$ denote the set of terminals $s$ that violate~\refeq{3opt}(i) w.r.t.
$X_i$, i.e. $s\in X_i$ but $s\notin \{s_i\}\cup S_i\cup S_{ij}\cup S_{ik}$
(where $\{i,j,k\}=\{1,2,3\}$). If $Q_i = \emptyset$ then $(X_i,\bar X_i)$ is
already $\Pi_{a_i}$-separating, in which case we set $X_i' := X_i$.

Let $Q_i\ne\emptyset$. We construct the digraph $G_i$ from $G$ by contracting
$VG - X_i$ into a new terminal $z_i$. Arc capacities in $G_i$ are induced by
those in $G$ (and are denoted by~$c$ as before). This gives the network
$\calN_i: = (G_i, \set{s_i,z_i} \cup Q_i, c)$ which is Eulerian at all vertices
except, possibly, for $s_i$ and $z_i$. The current multiflow $F$ in $G$ induces
a multiflow $F_i$ in $G_i$ consisting of weighted $s_i$--$z_i$ and $z_i$--$s_i$
paths. Since $F$ saturates $\deltain_G(X_i)$ and $\deltaout_G(X_i)$, the
multiflow $F_i$ saturates $\deltain_{G_i}(z_i)$ and $\deltaout_{G_i}(z_i)$.

Now we find in $G_i$ a maximum integer flow $g_i$ from the source $s_i$ to the
set of sinks $Q_i\cup z_i$. Moreover, among such flows we choose one maximizing
$-\div_{g_i}(z_i)$. (This is done by standard flow techniques: take a maximum
$s_i$--$z_i$ flow (e.g. by extracting the subflow in $F_i$ formed by
$s_i$--$z_i$ paths), then switch to the residual network and augment the
current flow to get a maximum $s_i$--$(Q_i\cup z_i)$ flow.) The flow $g_i$ is
decomposed into a collection of weighted $s_i$--$z_i$ paths, denoted by $\hat
g_i(s_i,z_i)$, a collection of weighted $s_i$--$t$ paths for $t \in Q_i$,
denoted by $\hat g_i(s_i,t)$, ignoring possible cycles. By the construction,
$g_i$ saturates the trivial cut $\deltain_{G_i}(z_i)$ and some $(s_i,Q_i\cup
z_i)$-cut $\deltaout_{G_i}(X_i')$.

It remains to construct flows on paths going in the opposite direction, i.e.
entering $s_i$. Define the function $h_i := c - g_i$ on $AG_i$. It is Eulerian
at all vertices in $VG_i - \left(\set{s_i,z_i} \cup Q_i\right)$. Also
$\div_{h_i}(z_i) \ge 0$ (since $g_i$ saturates $\deltain_{G_i}(z_i)$) and
$\div_{h_i}(t) \ge 0$ for all $t \in Q_i$ (since $c$ is Eulerian at $t$). We
decompose $h_i$ into a collection of weighted $z_i$--$s_i$ paths, denoted by
$\hat h_i(z_i,s_i)$), and a collection of weighted $t$--$s_i$ paths for $t \in
Q_i$, denoted by $\hat h_i(t,s_i)$, ignoring possible cycles. These paths
saturate $\deltain_{G_i}(X_i')$ (since $h_i(\deltaout_{G_i}(X_i')) = 0$) and
$\deltaout_{G_i}(z_i)$ (since $h_i(\deltain_{G_i}(z)) = 0$).

The collections $\hat g_i(\cdot)$ and $\hat h_i(\cdot)$ constitute a multiflow
$F'_i$ that replaces the ``restriction'' $F_i$ of $F$ on $G_i$. Making such
``replacements'' for $i=1,2,3$ (and using the fact that $X_1,X_2,X_3$ are
disjoint), we obtain an integer multiflow $F'$ which along with
$X'_1,X'_2,X'_3$ as above satisfies~\refeq{3opt}. Hence $F'$ is optimal.

\subsection{Complexity}

In this final section we describe an efficient implementation of our algorithm
and estimate its complexity. Current multiflows in the process are stored as
collections of point-to-point flows. Namely, an integer multiflow $F$ in a
network with terminals $S$ is maintained as a collection $\setst{f_{st}}{s, t
\in S, s \ne t}$, where $f_{st}$ is an integer $s$--$t$ flow (called an
\emph{$s$--$t$ component} of $F$).

Let $\phi(n,m,k)$ denote the complexity of the algorithm applied to an instance
with $n$ vertices and $m$ arcs of $G$, and $k$ leaves of~$T$.

The case $k = 2$ was studied in \refsubsec{two_leaves}. The algorithm involves
a single max-flow computation and two flow decompositions. Hence
\begin{equation}
\label{eq:complexity_2}
    \phi(n,m,2) = O(MF(n,m) + mn),
\end{equation}
where $MF(n',m')$ denotes the complexity of a max-flow algorithm in a network
with $n'$ vertices and $m'$ arcs.

The case $k = 3$ was considered in \refsubsec{three_leaves}. It reduces to
solving a three-terminal version of the unweighted directed \textbf{IMF}
problem followed by $O(1)$ max-flow computations and decompositions. With the
help of the algorithm from \cite{BK-07} the three-terminal multiflow problem is
solved in $O(MF(n,m) + mn \log (n^2/m))$ time. Therefore,
 \begin{equation}
\label{eq:complexity_3}
    \phi(n,m,3) = O(MF(n,m) + mn \log (n^2/m)).
\end{equation}

For $k \ge 4$, we apply the partitioning operation from \refsubsec{partition}.
Computing a minimum cut dividing the current instance $\calN$ into $\calN_1$
and $\calN_2$ takes $O(MF(n,m))$ time. The aggregation takes the $s$--$z_2$
components of $F_1$ (for $s \in S \cap X_1$) and the $z_1$--$s$ components of
$F_2$ (for $s \in S \cap X_2$), combines them into an $(S \cap X_1)$--$(S \cap
X_2)$ flow and decomposes it into a collection of flows for all source-sink
pairs. The algorithm similarly handles the $z_2$--$s$ components of $F_1$ (for
$s \in S \cap X_1$) and the $s$--$z_1$ components of $F_2$ (for $s \in S \cap
X_2$). Finally it adds remaining components of $F_1$ and $F_2$, thus forming an
integer optimal multiflow in $\calN$. In total the aggregation operations take
$O(mn)$ time, hence
 \begin{equation}
\label{eq:complexity_partition}
    \phi(n,m,k) = \phi(n_1,m_1,k_1) + \phi(n_2,m_2,k_2) + O(MF(n,m) + mn),
 \end{equation}
where $(n_i,m_i,k_i)$ are the size parameters for $\calN_i$.

Since degrees of inner nodes of $T$ are 3 by (C3), there exists (and can be
found in $O(k)$ time) a partitioning edge in $T$ that yields $k_1, k_2 \le 2k/3
+ 1$. Thus, the height of the recursion tree is at most $O(\log S)$. Also $n_1
+ n_2 = n + 2$ and $m_1, m_2 \le m$. Assuming that $MF(n,m)$ obeys some
technical conditions (e.g., satisfying $MF(n,m) = O(mn \log (n^2/m))$ as in the
algorithm of Goldberg and Tarjan \cite{GT-88}), one can show by induction that
\refeq{complexity_2}, \refeq{complexity_3}, and \refeq{complexity_partition}
imply
 $$
    \phi(n,m,k) = O((MF(n,m) + mn) \log k + mn \log(n^2/m))
 $$
(applying reasonings similar to those in~\cite{IKN}). By spending additional
$O(mn \log \abs{S})$ time, one can convert the resulting integer optimal
multiflow into path-packing form, as exlained in \cite{BK-07}. In total, the
algorithm takes $O((MF(n,m) + mn) \log \abs{S} + mn \log (n^2/m))$ time, as
declared.

\end{document}